\documentclass[]{amsart}

\usepackage[english]{babel}
\usepackage[utf8]{inputenc}
\usepackage[T1]{fontenc}
\usepackage{amsmath, amsfonts, amssymb, amsthm}
\usepackage{color}
\usepackage{ulem}
\usepackage{tikz}

\newcommand{\CC}{\mathbb{C}}

\newcommand{\GG}{\mathbb{G}}
\newcommand{\NN}{\mathbb{N}}
\newcommand{\PP}{\mathbb{P}}
\newcommand{\QQ}{\mathbb{Q}}

\newcommand{\ZZ}{\mathbb{Z}}

\newcommand{\Oc}{\mathcal{O}}

\newcommand{\set}[1]{\left\{ #1 \right\}}
\newcommand{\setb}[1]{\left( #1 \right)}
\newcommand{\abs}[1]{\left| #1 \right|}

\DeclareMathOperator{\lcm}{lcm}

\newtheorem{mymasterthm}{notForUse}
\theoremstyle{definition}

\theoremstyle{plain}

\newtheorem{mythm}[mymasterthm]{Theorem}

\newtheorem{myprop}[mymasterthm]{Proposition}

\title{Yet another $ S $-unit variant of diophantine tuples}

\author[C. Fuchs]{Clemens Fuchs}
\author[S. Heintze]{Sebastian Heintze}

\address{University of Salzburg\newline
	\indent Department of Mathematics\newline
	\indent Hellbrunnerstr. 34 \newline
	\indent A-5020 Salzburg, Austria}
\email{clemens.fuchs@sbg.ac.at, sebastian.heintze@sbg.ac.at}

\begin{document}
	
	\maketitle
	
	
	\begin{abstract}
		We show that there are only finitely many triples of integers $ 0 < a < b < c $ such that the product of any two of them is the value of a given polynomial with integer coefficients evaluated at an $ S $-unit that is also a positive integer. The proof is based on a result of Corvaja and Zannier and thus is ultimately a consequence of the Schmidt subspace theorem.
	\end{abstract}
	
	\section{Introduction}
	
	An old problem studied by Diophantus is to find positive integers such that the product of any two increased by one is a perfect square. There has been a lot of work done on this and related problems in the last decade. We do not enter that history, but refer to \cite{dujella-web} instead where these results are collected and described. Our intention is to add another point of view that was touched in some recent papers but left out so far.
	
	Let $ S = \set{p_1,\ldots,p_h} $ be a given set of rational primes that we will fix from now on. The ring of $ S $-integers will be denoted by $ \Oc_S = \{ p/q \in \QQ : p \in \ZZ, q = p_1^{k_1} \cdots p_h^{k_h}, k_1,\ldots,k_h \in \NN \} $ and its unit group, the group of $ S $-units, by $ \Oc_S^* $. We call an $ n $-tuple $ (a_1,\ldots,a_n) \in \ZZ^n $ with $ 0 < a_1 \leq \cdots \leq a_n $ an $ S $-Diophantine $ n $-tuple, if for all $ 1 \leq i < j \leq n $ we have $ a_i a_j + 1 = s_{i,j} $ is an $ S $-unit. Observe that if $ 2 \in S $ then we can find infinitely many $ S $-Diophantine $ n $-tuples of the form $ (1,\ldots,1,a) $. We will call such $ n $-tuples trivial.
	
	We first mention that Corvaja and Zannier proved in \cite{corvaja-zannier-2003} that there are at most finitely many $ S $-Diophantine triples $ (a,b,c) $ with $ a<b<c $. Actually, they proved that there are only finitely many triples of positive integers $ a<b<c $ such that the product $ (ab+1)(ac+1) $ has all of its prime factors in $ S $. (This led to a proof of a conjecture by Gy\H{o}ry, Sark\"ozy and Stewart on the largest prime divisor of $ (ab+1)(ac+1) $.)
	
	Let $ S = \set{p} $ consist just of one prime $ p $. It is easy to show that then no $ S $-Diophantine triple exists at all. To see this let, more generally, $ q $ be a positive integer and $ (a,b,c) \in \ZZ^3 $ with $ 0 < a \leq b \leq c $ such that $ ab+1=q^k, ac+1=q^m, bc+1=q^n $ for $ k,m,n \in \NN $. We have $ k \leq m \leq n $ and $ b-a = bq^m - aq^n $. This implies that $ (b-a)/(ab+1) \in \ZZ $ which is impossible unless $ a=b $. However, the only perfect powers that differ just by $ 1 $ are given by the example $ 3^2-2^3=1 $ (this is the famous Catalan conjecture that was solved be Mih\u ailescu in \cite{mihailescu-2004}), and this does not lead to a valid solution for $ q^k - a^2 = 1 $ either. In conclusion we assume from now on that $ \abs{S} \geq 2 $.
	
	We mention that this definition is motivated by \cite{fuchs-luca-szalay-2008} where the same kind of question was asked for a given linear combination of two perfect powers with given base instead of requiring that the product of any two plus one takes products of primes in a given set only (in fact there it was proved that there are only finitely many triples of distinct positive integers which take values in a given binary linear recurring sequence unless the recurrence is of certain exceptional shape).
	
	In this paper we give an extension of the result of Corvaja and Zannier. For this, let $ f \in \ZZ[X] $ be a non-constant polynomial. We are interested in $ (a_1, \ldots, a_n) \in \ZZ^n $ with $ 0 < a_1 \leq \cdots \leq a_n $ such that $ a_i a_j = f(s_{i,j}) $ for some $ s_{i,j} \in \Oc_S^* $ for all $ 1 \leq i < j \leq n $. Observe that $ f(X) = X-1 $ is the case discussed above.
	
	The theorem below is a generalisation of the result of Corvaja and Zannier \cite{corvaja-zannier-2003} and uses another result due to these authors. The theorem of course also implies that the size $ n $ such that there exists an $ S $-Diophantine $ n $-tuple $ (a_1,\ldots,a_n) $ is bounded. Clearly this bound depends on $ \abs{S} $ and cannot expected to be small; for example we have that $ (99,315,9920,32768,44460,19534284) $ has the property that the product of any two plus $ 2985984 $ is an $ S $-unit for $ S = \{ 2, 3, 5, 13, 19, 83, 103, 151, 163, 193, $ $ 199, 229, 283, 439, 463, 1019, 1453, 8629 \} $ (i.e. we have $ f(X) = X-2985984 $ with the notation below). Observe that another related result can be found in \cite{gyarmati-2005}.
	
	In the theorem below we consider only triples. If the reader is interested in issues about quadruples we refer to some results published by Luca, Szalay and Ziegler.
	In \cite{szalay-ziegler-2013-2} Szalay and Ziegler show that for $ S = \set{p,q} $ with $ p,q \equiv 3 \ \mathrm{mod}\ 4 $ no $ S $-Diophantine quadruple exists.
	In \cite{szalay-ziegler-2015} the same authors prove that if $ S = \set{2,q} $ with $ q \equiv 3 \ \mathrm{mod}\ 4 $ or $ q $ small enough also no $ S $-Diophantine quadruple exists.
	Further they show in \cite{szalay-ziegler-2013-1} the same result for any set $ S $ of cardinality two satisfying some technical conditions.
	Luca and Ziegler prove in \cite{luca-ziegler-2014} an upper bound for the number of $ S $-Diophantine quadruples depending on the cardinality of $ S $. This bound depends on an upper bound for the number of non-degenerate solutions to an $ S $-unit equation.
	
	There are obviously infinitely many pairs $ (a,b) $ with $ 0 < a < b $ satisfying $ ab = f(u) $ for $ u \in \Oc_S^* \cap \NN $ and a given $ f \in \ZZ[X] $, with positive leading coefficient, namely of the form $ (a,b) = (1,f(u)) $.
	So the natural next question is what we can say about triples.
	For the sake of an example $ (1,5,11) $ has the property that the product of any two is the value of $ f(X) = X^2-X-1 $ evaluated at an $ u \in \Oc_S^* \cap \NN $ for $ S = \set{2,3} $.
	
	Conversely, let us fix a triple $ (a,b,c) $, a set $ S $ of primes and natural numbers $ u,v,w \in \Oc_S^* $. Then one can construct a polynomial $ g \in \QQ[X] $ such that $ ab = g(u), ac = g(v), bc = g(w) $ via the Lagrange interpolation polynomial. Let now $ d $ be a common multiple of the denominators of the coefficients of $ g $. For $ f = d^2 g $ we get $ adbd = f(u), adcd = f(v), bdcd = f(w) $ for $ (ad,bd,cd) \in \ZZ^3 $ and $ f \in \ZZ[X] $. The last statement holds also true if we replace $ f $ by $ f+h $ for a polynomial $ h \in \ZZ[X] $ with $ h(u) = h(v) = h(w) = 0 $. Therefore for any $ S $ and any bound $ D $ we can find a polynomial $ f $ over the integers of degree at least $ D $ such that there exists at least one triple $ (a,b,c) \in \ZZ^3 $ satisfying $ 0 < a < b < c $ as well as $ ab = f(u), ac = f(v), bc = f(w) $ with $ u,v,w \in \Oc_S^* \cap \NN $.
	
	\section{Results and notations}
	
	The main theorem that we are going to prove in the next section is the following statement:
	
	\begin{mythm}
		\label{p1-thm:mainthm}
		Let $ S $ be a finite set of primes and $ f \in \ZZ[X] $ be a non-constant polynomial with $ f(0) \neq 0 $ and with at least one zero of odd multiplicity.
		Then there are at most finitely many triples $ (a,b,c) \in \ZZ^3 $ with $ 0 < a < b < c $ such that $ ab = f(u), ac = f(v), bc = f(w) $ with $ u,v,w \in \Oc_S^* \cap \NN $.
	\end{mythm}
	
	Let us first give some remarks on the assumptions of this theorem.
	We cannot remove the condition $ f(0) \neq 0 $. As a counterexample in this case consider the polynomial $ f(X) = X $ and an arbitrary finite (non-empty) set $ S $ of primes. Then it holts that $ f(0) = 0 $ but all the other assumptions of the theorem are satisfied. In that situation there exist obviously infinitely many triples $ (a,b,c) \in \ZZ^3 $ fulfilling $ a,b,c \in \Oc_S^* $ and $ 0<a<b<c $. Each such triple satisfies $ ab = u \in \Oc_S^* \cap \NN, ac = v \in \Oc_S^* \cap \NN, bc = w \in \Oc_S^* \cap \NN $.
	
	Furthermore the restriction $ u,v,w \in \Oc_S^* $ is necessary. If we only require them to be natural numbers, there is the following counterexample: Consider $ f(X) = (X+1)^3 $. Thus all assumptions are satisfied. For each triple $ (a_0,b_0,c_0) \in \ZZ^3 $ with $ 0 < a_0 < b_0 < c_0 $ the conclusion of the theorem holds with $ (a,b,c) = (a_0^3,b_0^3,c_0^3) $ and $ u = a_0b_0-1, v = a_0c_0-1, w = b_0c_0-1 $.
	
	It is easy to see that all the assumptions are satisfied in the case $ f(X) = X-1 $. Therefore the theorem is applicable in that situation.
	
	There are different possibilities how to deal with the equations $ ab = f(u), ac = f(v), bc = f(w) $. First one can consider common divisors of $ f(v), f(w) $ or $ f(u), f(w) $. A second usage would be to multiply the three equations. We will use both of these options in the proof of Theorem \ref{p1-thm:mainthm} below.
	A third possibility could be to eliminate one variable. For instance we can deduce from $ ac = f(v), bc = f(w) $ the equations $ f(v)/a = c = f(w)/b $ and thus $ f(v) = (a/b) \cdot f(w) $. Then Proposition 3.1 in \cite{hajdu-sarkozy-2018-2} may help but it does not immediately imply finiteness. In combination with this one could try to use Lemma 2.2 in \cite{hajdu-sarkozy-2018-1}. The required system of equations can be obtained by $ cf(u) = bf(v) = af(w) = abc $ which yields $ acf(u) = abf(v) $ and $ bcf(u) = abf(w) $. In our proof, however, we will not use these two results.
	
	Before we prove the theorem let us state a result from \cite{corvaja-zannier-2005} that we are going to apply. We denote by $ \overline{\QQ} $ the set of algebraic numbers in $ \CC $, by $ M_{\QQ} $ the set of places of $ \QQ $, normalized such that the product formula holds, and by $ h $ the logarithmic Weil height.
	
	\begin{myprop}[Proposition 4 in \cite{corvaja-zannier-2005}]
		\label{p1-prop:corzann-gcd}
		Let $ r,s \in \overline{\QQ}[X] $ be two non-zero polynomials, not both vanishing at $ 0 $. Then for every $ \varepsilon > 0 $, all but finitely many solutions $ (u,v) \in (\Oc_S^*)^2 $ to the inequality
		\begin{equation}
			\label{p1-eq:propineq}
			\sum_{\nu \in M_{\QQ}} \min \set{0, \log \max \set{\abs{r(u)}_{\nu}, \abs{s(v)}_{\nu}}} < - \varepsilon \max \set{h(u), h(v)}
		\end{equation}
		are contained in finitely many translates of one-dimensional subgroups of $ \GG_m^2 $ (which can be effectively determined).
	\end{myprop}
	
	Since this proposition is based on the Schmidt subspace theorem our result is ineffective in the sense that the proof does not show how to bound the number of triples.
	
	The idea of the proof of our main theorem is now the following: We assume that there are infinitely many triples satisfying the conditions of the statement.
	Now we apply the proposition to $ f(v), f(w) $ and distinguish between two cases.
	The first case is that infinitely many triples are not solutions of the inequality \eqref{p1-eq:propineq} and is handled by rewriting the expressions in the inequality.
	The second one assumes that infinitely many triples lie in a fixed one-dimensional subgroup of $ \GG_m^2 $.
	Here we apply the proposition once again for $ f(u), f(w) $ and consider the analogous subcases.
	If infinitely many triples are not solutions of \eqref{p1-eq:propineq} we get the contradiction by an analysis of the growth of some variables.
	Provided we are again in a fixed one-dimensional subgroup of $ \GG_m^2 $ the argument considers the multiplicity of roots of a certain polynomial.
	
	We conclude this section by explaining some notations that will be used later.
	The expression $ x \ll y $ means that there exists a positive constant $ L $ such that $  x \leq Ly $ for any in the respective context admissible pair $ (x,y) $.
	Furthermore we use the Landau symbol $ F(x) = O(G(x)) $ to denote that there exists a constant $ M $ such that $ \abs{F(x)} \leq M \abs{G(x)} $ for $ x $ large enough.
	
	\section{Proof of the main theorem}
	
	\begin{proof}[Proof of Theorem \ref{p1-thm:mainthm}]
		Let $ S = \set{p_1, \ldots, p_h} $.
		Assume that there are infinitely many $ (a,b,c) $ with $ 0 < a < b < c $ such that $ ab = f(u), ac = f(v), bc = f(w) $ for $ u,v,w \in \Oc_S^* \cap \NN $. Then it follows that $ c \rightarrow \infty $.
		First note that for
		\begin{equation*}
			f(X) = f_n X^n + \ldots + f_1 X + f_0
		\end{equation*}
		we have that $ f_n > 0 $. Otherwise $ f(x) $ would have an upper bound on $ \NN $ and therefore $ c $ would be bounded.
		Since $ c \rightarrow \infty $ we have also $ v \rightarrow \infty $ and $ w \rightarrow \infty $. With at most finitely many exceptions we can assume that each solution $ (a,b,c) $ corresponds to both an unique value of $ v $ and an unique value of $ w $.
		Furthermore we can assume that $ u < v < w $.
		
		By Proposition \ref{p1-prop:corzann-gcd} applied with $ r = s = f $ and $ \varepsilon_1 = \frac{1}{4} $ we get
		\begin{equation*}
			\sum_{\nu \in M_{\QQ}} \min \set{0, \log \max \set{\abs{f(v)}_{\nu}, \abs{f(w)}_{\nu}}} \geq - \frac{1}{4} \max \set{h(v), h(w)}
		\end{equation*}
		for all $ w $ large enough and except for $ (v,w) \in (\Oc_S^*)^2 $ in finitely many translates of one-dimensional subgroups of $ \GG_m^2 $.
		Now we have
		\begin{align*}
			\sum_{\nu \in M_{\QQ}} \min \set{0, \log \max \set{\abs{f(v)}_{\nu}, \abs{f(w)}_{\nu}}}
			&= \log \prod_{\nu \in M_{\QQ}} \min \set{1, \max \set{\abs{ac}_{\nu}, \abs{bc}_{\nu}}} \\
			&= \log \prod_{p \in \PP} \max \set{\abs{ac}_p, \abs{bc}_p} \\
			&= \log \left( (\gcd \setb{ac, bc})^{-1} \right) \\
			&= \log \left( (\gcd \setb{f(v), f(w)})^{-1} \right)
		\end{align*}
		and
		\begin{align*}
			- \frac{1}{4} \max \set{h(v), h(w)}
			&= \min \set{- \frac{1}{4} \log H(v), - \frac{1}{4} \log H(w)} \\
			&= \log \min \set{H(v)^{- \frac{1}{4}}, H(w)^{- \frac{1}{4}}} \\
			&= \log \min \set{v^{- \frac{1}{4}}, w^{- \frac{1}{4}}}
			= \log \left( w^{- \frac{1}{4}} \right).
		\end{align*}
		This yields
		\begin{equation*}
			c \leq \gcd \setb{ac, bc} = \gcd \setb{f(v), f(w)} \leq w^{\frac{1}{4}}
		\end{equation*}
		for all $ w $ large enough and except for $ (v,w) \in (\Oc_S^*)^2 $ in finitely many translates of one-dimensional subgroups of $ \GG_m^2 $.
		Since $ f(w) \leq c^2 \leq w^{\frac{1}{2}} $ ends up in a contradiction for large $ w $, infinitely many solutions must correspond to pairs $ (v,w) $ located in a fixed one-dimensional subgroup of $ \GG_m^2 $.
		
		From now on we have
		\begin{equation}
			\label{p1-eq:firstsubgroup}
			v^k w^{\ell} = g
		\end{equation}
		for fixed integers $ k,\ell $ with $ \gcd \setb{k, \ell} = 1 $ and $ k > 0 $ as well as a fixed $ g \in \QQ^* $. This implies $ \ell < 0 $.
		As $ v,w \in \Oc_S^* $ we can write them as $ v = \prod_{i=1}^{h} p_i^{v_i} $ and $ w = \prod_{i=1}^{h} p_i^{w_i} $.
		Thus $ g = \prod_{i=1}^{h} p_i^{g_i} $.
		For $ i = 1, \ldots, h $ this implies $ p_i^{k v_i + \ell w_i} = p_i^{g_i} $ and $ k v_i + \ell w_i = g_i $.
		Let us now consider a fixed pair $ (\overline{v}, \overline{w}) $. Using the analogous notation we have $ k \overline{v_i} + \ell \overline{w_i} = g_i $.
		Rewriting yields $ k (\overline{v_i} - v_i) = \ell (w_i - \overline{w_i}) $ and thanks to the fact that $ k $ and $ \ell $ are coprime and non-zero we get
		\begin{align*}
			w_i &= \overline{w_i} + k \mu_i \\
			v_i &= \overline{v_i} - \ell \mu_i.
		\end{align*}
		With $ \rho = \prod_{i=1}^{h} p_i^{\mu_i} $ we end up with $ w = \overline{w} \rho^k $ and $ v = \overline{v} \rho^{-\ell} $.
		These representations give us the following estimates:
		\begin{align*}
			c \leq f(v) &\ll v^n \ll \rho^{-\ell n} \\
			c \leq f(w) &\ll w^n \ll \rho^{k n} \\
			f(v) &\gg v^n \gg \rho^{-\ell n} \\
			f(w) &\gg w^n \gg \rho^{k n}.
		\end{align*}
		
		By applying Proposition \ref{p1-prop:corzann-gcd} again with an $ \varepsilon_2 $ to be fixed later we get
		\begin{equation*}
			\sum_{\nu \in M_{\QQ}} \min \set{0, \log \max \set{\abs{f(u)}_{\nu}, \abs{f(w)}_{\nu}}} \geq - \varepsilon_2 \max \set{h(u), h(w)}
		\end{equation*}
		and in the same way as above
		\begin{equation*}
			b \leq \gcd \set{ab, bc} = \gcd \set{f(u), f(w)} \leq w^{\varepsilon_2}
		\end{equation*}
		for all $ w $ large enough and except for $ (u,w) \in (\Oc_S^*)^2 $ in finitely many translates of one-dimensional subgroups of $ \GG_m^2 $.
		
		In the first case we distinguish between three subcases. First let us assume that $ k + \ell > 0 $.
		Here we choose $ \varepsilon_2 = \frac{(k+\ell)n}{2k} $ and get
		\begin{equation*}
			\rho^{(k+\ell)n} = \frac{\rho^{kn}}{\rho^{-\ell n}} \ll \frac{f(w)}{c} = b \leq w^{\varepsilon_2} \ll \rho^{k\varepsilon_2} = \rho^{(k+\ell)n/2}
		\end{equation*}
		which is a contradiction for large $ w $ as a large $ w $ would imply also a large $ \rho $.
		Next we assume $ k + \ell < 0 $. We choose $ \varepsilon_2 = \frac{-(k+\ell)n}{2k} $, get
		\begin{equation*}
			\rho^{-(k+\ell)n} = \frac{\rho^{-\ell n}}{\rho^{kn}} \ll \frac{f(v)}{c} = a < b \leq w^{\varepsilon_2} \ll \rho^{k\varepsilon_2} = \rho^{-(k+\ell)n/2}
		\end{equation*}
		and this is again a contradiction for large $ w $.
		Last but not least let $ k + \ell = 0 $ and choose $ \varepsilon_2 = \frac{1}{2} $. Since $ \gcd \setb{k, \ell} = 1 $ and $ k > 0 $ it is $ k = -\ell = 1 $.
		Therefore we have $ \frac{w}{v} = \phi $ for a positive constant $ \phi \in \QQ $.
		As $ \phi = 1 $ leads to $ a = \frac{f(v)}{c} = \frac{f(w)}{c} = b $, we can assume $ \phi \neq 1 $.
		Rewriting $ w = \phi v $ we get
		\begin{equation*}
			\frac{b}{a} = \frac{bc}{ac} = \frac{f(\phi v)}{f(v)} \overset{v \rightarrow \infty}{\longrightarrow} \phi^n = \frac{p^n}{q^n}
		\end{equation*}
		for integers $ p,q $ satisfying $ \phi = \frac{p}{q} $.
		Let us first suppose that there are infinitely many $ v $ corresponding to solutions such that $ \abs{bq^n - ap^n} < 1 $.
		As the absolute value is taken from an integer it follows $ bq^n = ap^n $ so $ b = a\phi^n $ and furthermore $ f(\phi v) = f(w) = bc = \phi^n ac = \phi^n f(v) $.
		Since this holds for infinitely many $ v $ the polynomial identity $ f(\phi X) = \phi^n f(X) $ must hold in $ \QQ[X] $. With $ \phi > 0, \phi \neq 1 $ it is an immediate consequence that the polynomial $ f(X) $ is of the form $ f_n X^n $ which is a contradiction to our assumptions in the theorem.
		Thus all but finitely many $ v $ correspond to solutions with $ \abs{bq^n - ap^n} \geq 1 $. Here we get
		\begin{align*}
			\frac{1}{aq^n}
			&\leq \abs{\frac{b}{a} - \frac{p^n}{q^n}}
			= \abs{\frac{b}{a} - \phi^n}
			= \abs{\frac{f(\phi v)}{f(v)} - \phi^n} \\
			&= \abs{\frac{f_n \phi^n v^n + f_{n-1} \phi^{n-1} v^{n-1} + \cdots + f_1 \phi v + f_0}{f_n v^n + f_{n-1} v^{n-1} + \cdots + f_1 v + f_0} - \phi^n} \\
			&= \abs{\frac{\phi^n + \frac{f_{n-1}}{f_n} \frac{\phi^{n-1}}{v} + \cdots + \frac{f_1}{f_n} \frac{\phi}{v^{n-1}} + \frac{f_0}{f_n} \frac{1}{v^n}}{1 + \frac{f_{n-1}}{f_n} \frac{1}{v} + \cdots + \frac{f_1}{f_n} \frac{1}{v^{n-1}} + \frac{f_0}{f_n} \frac{1}{v^n}} - \phi^n} \\
			&= \abs{\frac{0 + \frac{f_{n-1}}{f_n} \frac{\phi^{n-1}-\phi^n}{v} + \cdots + \frac{f_1}{f_n} \frac{\phi-\phi^n}{v^{n-1}} + \frac{f_0}{f_n} \frac{1-\phi^n}{v^n}}{1 + \frac{f_{n-1}}{f_n} \frac{1}{v} + \cdots + \frac{f_1}{f_n} \frac{1}{v^{n-1}} + \frac{f_0}{f_n} \frac{1}{v^n}}} \\
			&= \frac{1}{v} \cdot \abs{\frac{\frac{f_{n-1}}{f_n} \left( \phi^{n-1}-\phi^n \right) + O\left( \frac{1}{v} \right)}{1 + O\left( \frac{1}{v} \right)}}
			\leq \frac{B}{v}
		\end{align*}
		for a constant $ B $.
		Therefore we have $ v \ll a $ and
		\begin{equation*}
			\rho \ll v \ll a < b \leq w^{\varepsilon_2} \ll \rho^{\frac{1}{2}}
		\end{equation*}
		is once again a contradiction for large $ w $.
		In summary all three subcases end up in a contradiction.
		
		So we are again in the exceptional case which means that infinitely many solutions correspond to pairs $ (u,w) $ in a fixed one-dimensional subgroup of $ \GG_m^2 $.
		From here on we have in addition to equation \eqref{p1-eq:firstsubgroup} the equation
		\begin{equation}
			u^m w^r = s
		\end{equation}
		for fixed integers $ m,r $ with $ \gcd \setb{m,r} = 1 $ and $ m > 0 $ as well as a fixed $ s \in \QQ^* $.
		In the same manner as above we have $ u = \prod_{i=1}^{h} p_i^{u_i} $ and $ s = \prod_{i=1}^{h} p_i^{s_i} $.
		Furthermore for a fixed pair $ (\overline{u}, \overline{w}) $ in an analogous way it follows that $ m u_i + r w_i = s_i $ and $ m \overline{u_i} + r \overline{w_i} = s_i $ as well as $ m (\overline{u_i} - u_i) = r (w_i - \overline{w_i}) $.
		In this connection $ \overline{w} $ from above can be chosen such that we can use it here again.
		Let us first assume that $ r \neq 0 $.
		Using that $ m $ and $ r $ are coprime we get
		\begin{align*}
			w_i &= \overline{w_i} + m \tau_i \\
			u_i &= \overline{u_i} - r \tau_i.
		\end{align*}
		Combining both representations of $ w_i $ yields $ k\mu_i = m\tau_i = \lcm \setb{k,m} \cdot \psi_i $ where $ \mu_i = \psi_i \widetilde{k} $ and $ \tau_i = \psi_i \widetilde{m} $ for $ k \widetilde{k} = \lcm \setb{k,m} = m \widetilde{m} $.
		With $ x = -r \widetilde{m} $, $ y = -\ell \widetilde{k} $ and $ z = k \widetilde{k} = m \widetilde{m} $ we have
		\begin{align*}
			w_i &= \overline{w_i} + z \psi_i \\
			v_i &= \overline{w_i} + y \psi_i \\
			u_i &= \overline{u_i} + x \psi_i.
		\end{align*}
		In the case $ r = 0 $ we get the same representation using the definitions $ x = 0, y = -\ell, z = k $ and $ \psi_i = \mu_i $.
		It holds that $ z>0, y>0 $ and $ x \geq 0 $.
		Overall we have
		\begin{align*}
			u &= \prod_{i=1}^{h} p_i^{u_i} = \prod_{i=1}^{h} p_i^{\overline{u_i} + x \psi_i} = \overline{u} \left( \prod_{i=1}^{h} p_i^{\psi_i} \right)^x = \overline{u} \eta^x \\
			v &= \prod_{i=1}^{h} p_i^{v_i} = \prod_{i=1}^{h} p_i^{\overline{v_i} + y \psi_i} = \overline{v} \left( \prod_{i=1}^{h} p_i^{\psi_i} \right)^y = \overline{v} \eta^y \\
			w &= \prod_{i=1}^{h} p_i^{w_i} = \prod_{i=1}^{h} p_i^{\overline{w_i} + z \psi_i} = \overline{w} \left( \prod_{i=1}^{h} p_i^{\psi_i} \right)^z = \overline{w} \eta^z
		\end{align*}
		for $ \eta = \prod_{i=1}^{h} p_i^{\psi_i} $.
		
		The next step is to \glqq optimize\grqq\ the exponents for later. Let $ \psi_i = 3t_i + \beta_i $ for $ \beta_i \in \set{0,1,2} $.
		Then for infinitely many solutions the $ \beta_i $ are fixed. Therefore
		\begin{align*}
			u &= \overline{u} \left( \prod_{i=1}^{h} p_i^{\psi_i} \right)^x = \overline{u} \left( \prod_{i=1}^{h} p_i^{\beta_i} \right)^x \left( \prod_{i=1}^{h} p_i^{t_i} \right)^{3x} = \eta_1 X^{d_1} \\
			v &= \overline{v} \left( \prod_{i=1}^{h} p_i^{\psi_i} \right)^y = \overline{v} \left( \prod_{i=1}^{h} p_i^{\beta_i} \right)^y \left( \prod_{i=1}^{h} p_i^{t_i} \right)^{3y} = \eta_2 X^{d_2} \\
			w &= \overline{w} \left( \prod_{i=1}^{h} p_i^{\psi_i} \right)^z = \overline{w} \left( \prod_{i=1}^{h} p_i^{\beta_i} \right)^z \left( \prod_{i=1}^{h} p_i^{t_i} \right)^{3z} = \eta_3 X^{d_3}
		\end{align*}
		with $ \eta_1 = \overline{u} \left( \prod_{i=1}^{h} p_i^{\beta_i} \right)^x, \eta_2 = \overline{v} \left( \prod_{i=1}^{h} p_i^{\beta_i} \right)^y, \eta_3 = \overline{w} \left( \prod_{i=1}^{h} p_i^{\beta_i} \right)^z $ as well as $ d_1 = 3x, d_2 = 3y, d_3 = 3z $ and $ X = \prod_{i=1}^{h} p_i^{t_i} $.
		Note that $ \eta_1 > 0 $, $ \eta_2 > 0 $ and $ \eta_3 > 0 $.
		Furthermore $ d_1 \leq d_2 \leq d_3 $ holds as for $ v $ large enough we have $ u < v < w $.
		By multiplying the original equations we get $ (abc)^2 = f(u) f(v) f(w) $ and for $ Y = abc $ there are infinitely many quasi-integral points on
		\begin{equation}
			\label{p1-eq:poly-quasiint}
			Y^2 = f(\eta_1 X^{d_1}) f(\eta_2 X^{d_2}) f(\eta_3 X^{d_3}).
		\end{equation}
		Therefore by Theorem 1 in \cite{leveque-1964} the polynomial on the right hand side of equation \eqref{p1-eq:poly-quasiint} has at most two zeros of odd multiplicity.
		Let us now factor the polynomial $ f $ over its splitting field into
		\begin{equation*}
			f(X) = f_n (X - \alpha_1) \cdots (X - \alpha_t) \times \Box
		\end{equation*}
		where the $ \alpha_i $ are the pairwise distinct roots of odd multiplicity and $ t \geq 1 $ by assumption.
		In this notation we use the symbol $ \Box $ as a wild card for a polynomial that is the square of another polynomial.
		To complete the proof we will show that the polynomial
		\begin{equation*}
			f(\eta_1 X^{d_1}) f(\eta_2 X^{d_2}) f(\eta_3 X^{d_3}) =
			C \prod_{j=1}^{3} \prod_{i=1}^{t} \left(X^{d_j} - \frac{\alpha_i}{\eta_j}\right) \times \Box
		\end{equation*}
		has at least three roots of odd multiplicity.
		To do so we are going to show that the polynomial with the $ \Box $ omitted has three simple roots.
		
		We consider four cases.
		In the first one we assume that $ d_1 = d_2 = d_3 $.
		W.l.o.g. we have $ \eta_1 \leq \eta_2 \leq \eta_3 $ and $ \abs{\alpha_1} \leq \abs{\alpha_2} \leq \cdots \leq \abs{\alpha_t} $.
		If $ \eta_2 = \eta_3 $ we get the contradiction $ v = w $. So $ \eta_2 < \eta_3 $ and for $ 1 \leq i \leq t $
		\begin{equation*}
			\abs{\frac{\alpha_1}{\eta_3}} < \abs{\frac{\alpha_1}{\eta_2}} \leq \abs{\frac{\alpha_i}{\eta_2}} \quad \text{and} \quad \abs{\frac{\alpha_1}{\eta_3}} < \abs{\frac{\alpha_1}{\eta_1}} \leq \abs{\frac{\alpha_i}{\eta_1}}
		\end{equation*}
		which implies that there are $ d_3 \geq 3 $ simple roots coming from the factor $ X^{d_3} - \frac{\alpha_1}{\eta_3} $.
		In the second case we assume that $ d_1 \leq d_2 < d_3 $.
		For a complex number $ \gamma \in \CC $ let $ \arg \gamma \in (0,2\pi] $ be the argument of $ \gamma $ (the complex angle).
		W.l.o.g. it holds that $ \arg \alpha_1 \leq \arg \alpha_2 \leq \cdots \leq \arg \alpha_t $ (in general another ordering as in the first case).
		Since $ \eta_j > 0 $ for $ j = 1,2,3 $ we have $ \arg \frac{\alpha_i}{\eta_j} = \arg \alpha_i $ for all $ i $ and all $ j $.
		The minimal argument occurring under the roots of $ (X^{d_1} - \frac{\alpha_1}{\eta_1}) \cdots (X^{d_1} - \frac{\alpha_t}{\eta_1}) (X^{d_2} - \frac{\alpha_1}{\eta_2}) \cdots (X^{d_2} - \frac{\alpha_t}{\eta_2}) $ is $ \frac{1}{d_2} \arg \alpha_1 $.
		The minimal argument occurring under the roots of $ (X^{d_3} - \frac{\alpha_1}{\eta_3}) \cdots (X^{d_3} - \frac{\alpha_t}{\eta_3}) $ is $ \frac{1}{d_3} \arg \alpha_1 $.
		Since $ \frac{1}{d_3} \arg \alpha_1 < \frac{1}{d_2} \arg \alpha_1 $ there is a simple root coming from the factor $ X^{d_3} - \frac{\alpha_1}{\eta_3} $.
		
		Consider now the following: Let $ \zeta $ be a common root of $ X^{d_{j_1}} - \frac{\alpha_{i_1}}{\eta_{j_1}} $ and $ X^{d_{j_2}} - \frac{\alpha_{i_2}}{\eta_{j_2}} $.
		Then $ \zeta e^{2\pi i/3} $ and $ \zeta e^{4\pi i/3} $ are also common roots of them since $ 3 \mid d_{j_1} $ and $ 3 \mid d_{j_2} $.
		Thus the number of common roots of $ X^{d_{j_1}} - \frac{\alpha_{i_1}}{\eta_{j_1}} $ and $ X^{d_{j_2}} - \frac{\alpha_{i_2}}{\eta_{j_2}} $ is a multiple of $ 3 $.
		The number of all zeros of each of them is also a multiple of $ 3 $. 
		Therefore the number of roots of $ X^{d_{j_1}} - \frac{\alpha_{i_1}}{\eta_{j_1}} $ that occur in no other $ X^{d_{j_2}} - \frac{\alpha_{i_2}}{\eta_{j_2}} $ is a multiple of $ 3 $.
		
		Using this there are at least three simple roots coming from the factor $ X^{d_3} - \frac{\alpha_1}{\eta_3} $.
		The third case assumes that $ 0 = d_1 < d_2 = d_3 $.
		W.l.o.g. we have again $ \eta_2 \leq \eta_3 $ and $ \abs{\alpha_1} \leq \abs{\alpha_2} \leq \cdots \leq \abs{\alpha_t} $.
		If $ \eta_2 = \eta_3 $ we get once again the contradiction $ v = w $. So $ \eta_2 < \eta_3 $ and for $ 1 \leq i \leq t $
		\begin{equation*}
			\abs{\frac{\alpha_1}{\eta_3}} < \abs{\frac{\alpha_1}{\eta_2}} \leq \abs{\frac{\alpha_i}{\eta_2}}
		\end{equation*}
		which implies that there are $ d_3 \geq 3 $ simple roots coming from the factor $ X^{d_3} - \frac{\alpha_1}{\eta_3} $.
		Last but not least we have to handle the case $ 0 < d_1 < d_2 = d_3 $.
		W.l.o.g. it holds that $ \arg \alpha_1 = \arg \alpha_2 = \cdots = \arg \alpha_{\lambda} < \arg \alpha_{\lambda+1} \leq \cdots \leq \arg \alpha_t $ and $ \abs{\alpha_1} < \abs{\alpha_2} < \cdots < \abs{\alpha_{\lambda}} $ and $ \eta_2 \leq \eta_3 $.
		Once more we have $ \arg \frac{\alpha_i}{\eta_j} = \arg \alpha_i $ for all $ i $ and all $ j $ as well as $ \eta_2 < \eta_3 $.
		The minimal argument occurring under the roots of $ (X^{d_1} - \frac{\alpha_1}{\eta_1}) \cdots (X^{d_1} - \frac{\alpha_t}{\eta_1}) $ is $ \frac{1}{d_1} \arg \alpha_1 $.
		The minimal argument occurring under the roots of $ (X^{d_2} - \frac{\alpha_1}{\eta_2}) \cdots (X^{d_2} - \frac{\alpha_{\lambda}}{\eta_2}) $ is $ \frac{1}{d_2} \arg \alpha_1 $.
		The minimal argument occurring under the roots of $ (X^{d_2} - \frac{\alpha_{\lambda+1}}{\eta_2}) \cdots (X^{d_2} - \frac{\alpha_t}{\eta_2}) $ is $ \frac{1}{d_2} \arg \alpha_{\lambda+1} $.
		The minimal argument occurring under the roots of $ (X^{d_3} - \frac{\alpha_1}{\eta_3}) \cdots (X^{d_3} - \frac{\alpha_{\lambda}}{\eta_3}) $ is $ \frac{1}{d_2} \arg \alpha_1 $.
		The minimal argument occurring under the roots of $ (X^{d_3} - \frac{\alpha_{\lambda+1}}{\eta_3}) \cdots (X^{d_3} - \frac{\alpha_t}{\eta_3}) $ is $ \frac{1}{d_2} \arg \alpha_{\lambda+1} $.
		Since $ \frac{1}{d_1} \arg \alpha_1 > \frac{1}{d_2} \arg \alpha_1 $ and $ \frac{1}{d_2} \arg \alpha_{\lambda+1} > \frac{1}{d_2} \arg \alpha_1 $ there are exactly $ 2\lambda $ zeros with argument $ \frac{1}{d_2} \arg \alpha_1 $.
		They have the absolut values
		\begin{equation*}
			\abs{\frac{\alpha_1}{\eta_3}}^{1/d_2}, \ldots, \abs{\frac{\alpha_{\lambda}}{\eta_3}}^{1/d_2}, \abs{\frac{\alpha_1}{\eta_2}}^{1/d_2}, \ldots, \abs{\frac{\alpha_{\lambda}}{\eta_2}}^{1/d_2}.
		\end{equation*}
		Therefore there is only one zero with argument $ \frac{1}{d_2} \arg \alpha_1 $ and absolute value $ \abs{\frac{\alpha_1}{\eta_3}}^{1/d_2} $.
		Using the consideration above there are at least three simple roots coming from the factor $ X^{d_3} - \frac{\alpha_1}{\eta_3} $.
		Now we have all cases finished and the statement is proven.
	\end{proof}


\begin{thebibliography}{99}
		\bibitem{corvaja-zannier-2003}
			\textsc{P. Corvaja and U. Zannier},
			On the greatest prime factor of $ (ab+1)(ac+1) $,
			\textit{Proc. Amer. Math. Soc.} \textbf{131} (2003), 1705-1709.
		\bibitem{corvaja-zannier-2005}
			\textsc{P. Corvaja and U. Zannier},
			A lower bound for the height of a rational function at $ S $-unit points,
			\textit{Monatsh. Math.} \textbf{144} (2005), 203-224.
		\bibitem{dujella-web}
			\textsc{A. Dujella},
			\textsf{https://web.math.pmf.unizg.hr/$\sim$duje/dtuples.html}
		\bibitem{fuchs-luca-szalay-2008}
			\textsc{C. Fuchs, F. Luca and L. Szalay},
			Diophantine triples with values in binary recurrences,
			\textit{Ann. Sc. Norm. Super. Pisa Cl. Sc. (5)} \textbf{7} (2008), 579-608.
		\bibitem{gyarmati-2005}
			\textsc{K. Gyarmati},
			A polynomial extension of a problem of Diophantus,
			\textit{Publ. Math. Debrecen} \textbf{66} (2005), 389-405.
		\bibitem{hajdu-sarkozy-2018-1}
			\textsc{L. Hajdu and A. Sark\"ozy},
			On multiplicative decompositions of polynomial sequences, I,
			\textit{Acta Arith.} \textbf{184} (2018), 139-150.
		\bibitem{hajdu-sarkozy-2018-2}
			\textsc{L. Hajdu and A. Sark\"ozy},
			On multiplicative decompositions of polynomial sequences, II,
			\textit{Acta Arith.} \textbf{186} (2018), 191-200.
		\bibitem{leveque-1964}
			\textsc{W. J. LeVeque},
			On the equation $ y^m = f(x) $,
			\textit{Acta Arith.} \textbf{9} (1964), 209-219.
		\bibitem{luca-ziegler-2014}
			\textsc{F. Luca and V. Ziegler},
			A note on the number of $ S $-Diophantine quadruples,
			\textit{Comm. in Math.} \textbf{22} (2014), 49-55.
		\bibitem{mihailescu-2004}
			\textsc{P. Mih\u ailescu},
			Primary cyclotomic units and a proof of Catalan's conjecture,
			\textit{J. Reine Angew. Math.} \textbf{572} (2004), 167-195.
		\bibitem{szalay-ziegler-2013-1}
			\textsc{L. Szalay and V. Ziegler},
			On an $ S $-unit variant of Diophantine $ m $-tuples,
			\textit{Publ. Math. Debrecen} \textbf{83} (2013), no. 1-2, 97-121.
		\bibitem{szalay-ziegler-2013-2}
			\textsc{L. Szalay and V. Ziegler},
			$ S $-Diophantine quadruples with two primes congruent $ 3 $ modulo $ 4 $,
			\textit{INTEGERS, Elec. J. of Comb. Num. Th.} \textbf{13} (2013), A80.
		\bibitem{szalay-ziegler-2015}
			\textsc{L. Szalay and V. Ziegler},
			$ S $-Diophantine quadruples with $ S = \set{2,q} $,
			\textit{Int. J. Num. Th.} \textbf{11} (2015), no. 3, 849-868.
	\end{thebibliography}
\end{document}